\documentclass[letterpaper,10pt,reqno]{amsart}

\usepackage{amsthm,amsmath,amssymb,amsfonts}
\usepackage[left=1.25in,right=1.25in,top=1.25in,bottom=1.5in, marginparwidth=1in]{geometry}
\usepackage[utf8x]{inputenc}
\usepackage{stmaryrd}
\usepackage{hyperref}
 \hypersetup{colorlinks=true,linktocpage=true,linkcolor=blue,citecolor=blue,urlcolor=blue}
   \hypersetup{
  pdftitle={Twin vector fields and independence of spectra},
  pdfauthor={Valente Ramirez}
}

\newcommand{\C}{\mathbb{C}}
\renewcommand{\P}{\mathbb{P}}
\newcommand{\A}{\mathcal{A}}
\newcommand{\Aff}[2]{\operatorname{Aff}({#1},{#2})}
\newcommand{\PGL}[2]{\operatorname{PGL}({#1},{#2})}
\newcommand{\tr}[1]{\operatorname{tr}\left({#1}\right)}
\renewcommand{\det}[1]{\operatorname{det}({#1})}
\newcommand{\spec}[2]{\operatorname{Spec}_{#1} {#2}}
\newcommand{\Ah}{\widehat{\mathcal{A}}}
\newcommand{\Sh}{\widehat{\mathcal{S}}}
\newcommand{\At}{\widetilde{\mathcal{A}}}
\renewcommand{\S}{\mathcal{S}}
\newcommand{\Sing}[1]{\operatorname{Sing}({#1})}

\newtheorem{theorem}{Theorem}
\newtheorem{corollary}[theorem]{Corollary}
\newtheorem{lemma}[theorem]{Lemma}
\newtheorem{proposition}[theorem]{Proposition}

\newtheorem{question}{Question}
  \theoremstyle{definition}
\newtheorem{definition}[theorem]{Definition}
  \theoremstyle{remark}

\theoremstyle{plain}
\newtheorem*{thmbis}{Theorem 1.3'}

\title[Twin vector fields and independence of spectra]{Twin vector fields and independence of spectra\\ for quadratic vector fields}
\author[V.~Ram\'{i}rez]	{Valente Ram\'{i}rez}
\address{Department of Mathematics\\Cornell University\\120 Malott Hall\\Ithaca NY 14853\\USA}
\email{valente@math.cornell.edu}
\date{\today}
\thanks{This work was supported by the grants UNAM-DGAPA-PAPIIT IN-102413 and CONACYT 219722.}

\begin{document}

\begin{abstract}
The object of this paper is to address the following question: When is a polynomial vector field on $\C^2$ completely determined (up to affine equivalence) by the spectra of its singularities? We will see that for quadratic vector fields this is not the case: given a generic quadratic vector field there is, up to affine equivalence, exactly one other vector field which has the same spectra of singularities. Moreover, we will see that we can always assume that both vector fields have the same singular locus and at each singularity both vector fields have the same spectrum. Let us say that two vector fields are \emph{twin vector fields} if they have the same singular locus and the same spectrum at  each singularity. 

To formalize the above claim we shall prove the following: any two generic quadratic vector fields with the same spectra of singularities (yet possibly different singular locus) can be transformed by suitable affine maps to be either the same vector field or a pair of twin vector fields. 

We then analyze the case of quadratic Hamiltonian vector fields in more detail and find necessary and sufficient conditions for a collection of non-zero complex numbers to arise as the spectra of singularities of a quadratic Hamiltonian vector field. 

Lastly, we show that a generic quadratic vector field is completely determined (up to affine equivalence) by the spectra of its singularities together with the characteristic numbers of its singular points at infinity.
\end{abstract}

\maketitle

\section{Introduction}\label{sec:intro}

In this work we will consider polynomial vector fields on the affine plane $\C^2$. Let us denote by $\A_n$ the vector space of all polynomial vector fields
\[ v=P(x,y)\frac{\partial}{\partial x}+Q(x,y)\frac{\partial}{\partial y}, \]
such that $P$ and $Q$ have degree at most $n$. By Bezout's theorem, a generic element of $\A_n$ has exactly $n^2$ isolated singularities; let us define then $\At_n$ to be the space of vector fields $v\in\A_n$ having $n^2$ isolated singularities. Throughout this paper we will consider exclusively vector fields from the classes $\At_n$. We say that two vector fields $v,w$ are \emph{affine equivalent} if there exists an affine map $T$ that transforms $v$ into $w$, that is, $w(x,y)=DT\cdot v(T^{-1}(x,y))$. Denote by $\Sing{v}$ the singular locus of $v$. If $p\in \Sing{v}$, we define the \emph{spectrum} of $v$ at $p$ as the (unordered) pair of eigenvalues of the linearization matrix
\[ Dv(p)=\begin{pmatrix}
          P'_x & P'_y\\
	  Q'_x & Q'_y
         \end{pmatrix}\bigg\vert_{(x,y)=p}. \]
Note that if $M\in\mathfrak{gl}_2\C$ then the spectrum of $A$ carries exactly the same information as the ordered pair 
\[ \spec{}{M}:=(\operatorname{tr}{M},\operatorname{det}{M}). \]
It will be more convenient for our purposes to think of the spectra as such ordered pairs.

If $X$ is any topological space and $m\geq1$, let $S_m$ denote the symmetric group on $m$ elements and $X^m/S_m$ the quotient of the usual action of $S_m$ on $X^m$ which permutes the components. The set of spectra of singularities of a generic polynomial vector field of degree $n$ belongs to the space
\[ \S_n:=(\C^2)^{n^2}/S_{n^2}, \]
which is an irreducible affine algebraic variety. We have a well-defined map
\[ \spec{n}{}\colon\At_n\to\S_n. \]

In very general terms, we call \emph{independence of spectra} the question of understanding the image and the fibers of the map $\spec{n}{}$. We will see below that in the case of quadratic vector fields we have a good understanding of this question: the closure of the image of $\spec{2}{}$ is a codimension-two subvariety defined by the \emph{Euler-Jacobi relations on spectra} (i.e.~equations (\ref{eq:EJ1}) and (\ref{eq:EJ2}) in \hyperref[coro:EJrelations]{Corollary \ref*{coro:EJrelations}}) and the fiber of $\spec{2}{}$ over a generic point in its image consists of two disjoint orbits of the action of the affine group $\Aff{2}{\C}$ on $\At_2$.

\begin{definition}\label{def:samespectra}
We say that two vector fields from the class $\At_n$ have \emph{the same spectra of singularities} if they have the same image under the map $\spec{n}{}$.
\end{definition}

Note that the above definition involves the spectra of the singularities and does not take into account the position of these. Of particular interest are pairs of vector fields that share both position and spectra of singularities.

\begin{definition}\label{def:twins}
We will say that two vector fields $v_1,v_2\in\At_n$ are \emph{twin vector fields} if they are not equal yet they have exactly the same singular locus and, for each point $p$ in the common singular set, the matrices $Dv_1(p)$ and $Dv_2(p)$ have the same spectrum.
\end{definition}

The next two propositions are the main results of this work.

\begin{theorem}\label{thm:main}
If two quadratic vector fields on $\C^2$ with non-degenerate singularities have the same spectra of singularities then, possibly after transforming one of them by a suitable affine map, they either agree or they are twin vector fields. Moreover, a generic quadratic vector field has exactly one twin vector field.
\end{theorem}

Note that in the generic case a pair of quadratic twin vector fields cannot be affine equivalent to each other hence the above theorem implies that given a generic vector field $v$ there exist exactly two disjoint orbits of the action of the affine group on $\At_2$ consisting of vector fields with the same spectra as $v$.

\begin{theorem}\label{thm:charnumbers}
A generic quadratic vector field is completely determined (up to affine equivalence) by the spectra of its finite singularities and the characteristic numbers of its singular points at infinity.
\end{theorem}

An important question about independence of spectra is the \emph{realization problem}: 

\begin{question}\label{ques:realization}
Which collections of numbers in $\S_n$ can be realized as the spectra of a degree $n$ polynomial vector field?
\end{question}

We shall answer this question in the particular case of quadratic Hamiltonian vector fields in \hyperref[thm:realization]{Theorem \ref*{thm:realization}}.

We remark that \hyperref[thm:main]{Theorem \ref*{thm:main}} is very similar in spirit to results in \cite{LinsNeto2012} (for foliations on $\P^2$ of degree two) and \cite{IlyashenkoMoldavskis2011} (for foliations on $\P^2$ coming from a generic quadratic vector field on $\C^2$), where it is proved that in the generic case the \emph{Baum-Bott indices} completely determine a foliation up to finite ambiguity (modulo the natural action of $\PGL{2}{\C}$ and $\Aff{2}{\C}$, respectively). In fact, Lins Neto proves that the generic fiber of the Baum-Bott map, which associates to a foliation the Baum-Bott indices of its singularities, contains exactly 240 orbits of the natural action of $\PGL{2}{\C}$.

This paper is organized as follows: we shall first discuss a few general facts about polynomial vector fields before specializing to the case of quadratic ones and proving \hyperref[thm:main]{Theorem \ref*{thm:main}}. After this, we will analize the case of quadratic Hamiltonian vector fields. We then discuss \hyperref[thm:charnumbers]{Theorem \ref*{thm:charnumbers}} and conclude with a related open question.

\section{The Euler-Jacobi formula}\label{sec:EJ}

Let us recall the statement of the Euler-Jacobi formula in the particular case relevant to us \cite[Chpt.~5, Sec.~2]{GriffithsHarris1994}.

\begin{theorem}\label{thm:EJformula}
If $P,Q$ are polynomials in $\C[x,y]$ of degree $n$ whose divisors intersect transversely in $n^2$ different points $p_1,\ldots,p_{n^2}\in\C^2$ and $g(x,y)$ is a polynomial of degree at most $2n-3$ then
\[ \sum_{k=1}^{n^2} \frac{g(p_k)}{\mathbf{J}(p_k)}=0, \]
where $\mathbf{J}(x,y)$ is the Jacobian determinant $\mathbf{J}(x,y)=\operatorname{det}\,\displaystyle\frac{\partial(P,Q)}{\partial(x,y)}$.
\end{theorem}

Consider a polynomial vector field $v=P\frac{\partial}{\partial x}+Q\frac{\partial}{\partial y}$ of degree $n\geq2$. By making $g(x,y)=1$ or $g(x,y)=\tr{Dv(x,y)}$ we obtain polynomials whose value at the singular point $p_k$ depens exclusively on the spectrum of $Dv(p_k)$.

\begin{corollary}\label{coro:EJrelations}
If $v=P\frac{\partial}{\partial x}+Q\frac{\partial}{\partial y}$ is a polynomial vector field of degree $n\geq2$ having $n^2$ non-degenerate singularities $p_1,\ldots,p_{n^2}\in\C^2$ then
\begin{align}
\sum_{k=1}^{n^2}\frac{1}{\det{Dv(p_k)}}=0, \label{eq:EJ1} \\
\sum_{k=1}^{n^2}\frac{\tr{Dv(p_k)}}{\det{Dv(p_k)}}=0. \label{eq:EJ2}
\end{align}
\end{corollary}
\noindent We call these equations the \emph{Euler-Jacobi relations on spectra}.

Note that because of the condition $\operatorname{deg }g(x,y)\leq 2n-3$ in the Euler-Jacobi formula, the only choices of $g$ that will give us relations on the spectra are $g(x,y)=1$ and $g(x,y)=\tr{Dv(x,y)}$.

\section{Counting dimensions}\label{sec:dimensions}

Let us exclude the trivial case $n=1$ in the following discussion. The space $\A_n$ has dimension $(n+1)(n+2)$. The affine group $\Aff{2}{\C}$ acts on $\A_n$ in a natural way and it is clear that affine equivalent vector fields have the same spectra of singularities. Let us define $\Ah_n=\At_n\sslash\Aff{2}{\C}$, that is, the quotient of this action in the senese of geometric invariant theory; we have
\[ \operatorname{dim}\,\Ah_n=n^2+3n-4. \]

On the other hand the map $\spec{n}{}$ was defined to have codomain $\S_n:=(\C^2)^{n^2}/S_{n^2}$, which has dimension $2n^2$. Let us denote by $\Sh_n$ the Zariski closure of the set of points in $\S_n$ with non-zero components that satisfy equations (\ref{eq:EJ1}) and (\ref{eq:EJ2}). In this way $\spec{n}{}$ actually takes values on $\Sh_n$. We have
\[ \operatorname{dim}\,\Sh_n=2n^2-2. \]

We obtain an induced map on the quotient $\widehat{\operatorname{Spec}}_n\colon\Ah_n\to\Sh_n$, and the above arguments show that $\operatorname{dim}\,\Sh_n\geq\operatorname{dim}\,\Ah_n$, with equality only in the case $n=2$. This can be rephrased by saying that the number of analytic invariants arising from the spectra (modulo the Euler-Jacobi relations on spectra) is at least as big as the number of parameters needed to define a vector field (up to affine equivalence). The case $n=2$, where both dimensions are equal, is analyzed in the next section; we will see that, modulo affine equivalence, the spectra determines the vector field up to finite ambiguity (in fact, up to its unique twin). This can be rephrased by saying that the generic fiber of the induced map 
\[ \widehat{\operatorname{Spec}}_2\colon\Ah_2\to\Sh_2 \] 
consists of two points -- or equivalently, that it is a generically two-to-one map.

Because of the above dimension count it is reasonable to expect that if $n>2$ then the spectra of singularities will turn out to be a complete set of analytic invariants.

\section{Quadratic vector fields}\label{sec:QVFs}

The next lemma is an essential step in the proof of \hyperref[thm:main]{Theorem \ref*{thm:main}}.

\begin{lemma}\label{lemma}
Let $v$ be a quadratic vector field having four non-degenerate singularities $p_1,\ldots,p_4$. The position and spectrum of $p_4$ is completely determined by the position and spectra of $p_1,p_2,p_3$.
\end{lemma}

\begin{proof}
This lemma is a double application of the Euler-Jacobi formula. First, we can think of relations (\ref{eq:EJ1}) and (\ref{eq:EJ2}) as a system of equations which we can solve for $\spec{}{Dv(p_4)}$ every time we are given $\spec{}{Dv(p_k)}$ for $k=1,2,3$. Second, if we let $g_1(x,y)=x$ and $g_2(x,y)=y$ in the Euler-Jacobi formula, then the system of equations 
\[ \sum_{k=1}^4 \frac{g_1(p_k)}{\det{Dv(p_k)}}=0,\qquad \sum_{k=1}^4 \frac{g_2(p_k)}{\det{Dv(p_k)}}=0, \]
determines completely the position of $p_4$.
\end{proof}

Note that the above lemma implies in particular that two quadratic vector fields are twins if and only if three out of their four singularities agree in position and spectra.

\begin{proof}[Proof of Theorem \ref*{thm:main}]
Suppose that two quadratic vector fields $v,\tilde{v}$ with non-degenerate singularities have the same spectra. After transforming $\tilde{v}$ by a suitable affine map on $\C^2$ we may assume that $v$ and $\tilde{v}$ have three singularities that agree in position and spectra. By \hyperref[lemma]{Lemma \ref*{lemma}} the same holds for the fourth singularity. We conclude that either $\tilde{v}=v$ or $\tilde{v}$ is a twin vector field of $v$.

We now prove that twin vector fields exist and are unique. Let $v=P\frac{\partial}{\partial x}+Q\frac{\partial}{\partial y}$ be a quadratic vector field having four non-degenerate singularities $p_1,\ldots,p_4$. By Max Noether's theorem, any quadratic polynomial $H$ which vanishes on the singular set $\Sing{v}=\{P=0\}\cap\{Q=0\}$ can be written uniquely as
\[ H=\alpha P+\beta Q, \]
for some complex numbers $\alpha,\beta$. This means that any quadratic vector field that vanishes on the singular set $\Sing{v}$ can be uniquely written as
\[ \tilde{v}=(aP+bQ)\frac{\partial}{\partial x}+(cP+dQ)\frac{\partial}{\partial y}, \]
for complex numbers $a,b,c,d$. Let us denote $A=\left(\begin{smallmatrix}a&b\\c&d\end{smallmatrix}\right)$ and note that $D\tilde{v}(x,y)=A\cdot Dv(x,y)$. In virtue of \hyperref[lemma]{Lemma \ref*{lemma}} the vector field $\tilde{v}$ has the same spectra as $v$ if and only if they have the same spectra at $p_1,p_2,p_3$. This happens if and only if
\begin{equation}\label{eq:system} \begin{array}{rcl}
\tr{A\cdot Dv(p_k)}&=&\tr{Dv(p_k)},\quad \mbox{for } k=1,2,3, \\
\det{A} &=& 1. \\
\end{array} \end{equation}
The above is a system of three linear equations and one quadratic equation on $a,b,c,d$. If the system is independent, we can always eliminate three of these variables using the linear equations and then substitute into the quadratic one. This gives a quadratic equation in one variable which generically would two different solutions. In order to check that for a generic quadratic vector field the linear system is independent and the discriminant of the quadratic equation is not zero it is enough to find a single example with these properties. A simple example is given by the Hamiltonian vector field $v=x(x+2y-1)\frac{\partial}{\partial x}+y(-2x-y+1)\frac{\partial}{\partial y}$. The computations are straightforward and we will omit them here.

This proves that for a generic vector field system (\ref{eq:system}) has two solutions; one corresponds to the original vector field $v$ and the other to a different vector field $\tilde{v}$, thus establishing existence and uniqueness of a twin vector field for a generic vector field $v$.
\end{proof}

\section{Quadratic Hamiltonian vector fields}\label{sec:QHVFs}

The quadratic Hamiltonian case can be very clearly understood. Let $\mathcal{H}_2$ denote the space of quadratic Hamiltonian vector fields. This space has dimension 9, since the space of polynomials in $\C[x,y]$ of degree three is ten dimensional. Alternatively, we can see this by noting that a quadratic vector field $v$ with non-degenerate singularities is Hamiltonian if and only if at every singular point $p$ we have $\tr{Dv(p)}=0$. Indeed, the linear polynomial $\tr{Dv(x,y)}=P'_x+Q'_y$ will be identically zero as soon as it vanishes on three non-collinear points. This imposes three independent conditions and so $\mathcal{H}_2$ has codimension three in the space of all quadratic vector fields, which has dimension 12. Because of this last argument any vector field that is affine equivalent to a Hamiltonian vector field with non-degenerate singularities is itself Hamiltonian. Note that the quotient $\mathcal{H}_2\sslash\Aff{2}{\C}$ has dimension three. On the other hand, each singularity of a Hamiltonian vector field carries only one analytic invariant: the determinant of its linearization matrix. These invariants are subject to the Euler-Jacobi relation (\ref{eq:EJ1}) and so the space of possible spectra is three dimensional. Once again, the dimension of the space of essential parameters matches the dimension of the space of invariants and the spectra map has its maximal possible rank. This implies the following result.

\begin{theorem}\label{thm:twinHam}
A generic quadratic Hamiltonian vector field has a unique twin vector field. In fact, the twin vector field of a generic quadratic Hamiltonian vector field $v$ is precisely its negative $-v$.
\end{theorem}

Because each singularity of a Hamiltonian vector field has vanishing trace it is clear that $v$ and $-v$ have the same singular set and the same spectra of singularities. Note however that in general these vector fields are not affine equivalent.

In the case of quadratic Hamiltonian vector fields we can explicitly say what we mean by \emph{generic} in the above theorem.

\begin{definition}
Let $v$ be a quadratic Hamiltonian vector field with non-degenerate singularities. We say that the spectrum of $v$ is \emph{exceptional} if there exist singular points $p_i,p_j$ such that
\[ \det{Dv(p_j)}+\det{Dv(p_k)}=0. \]
A quadratic Hamiltonian vector field is a \emph{generic quadratic Hamiltonian vector field} if it has non-degenerate singularities and its spectrum is not exceptional.
\end{definition}

Note that if $\det{Dv(p_j)}+\det{Dv(p_k)}=0$ then the eigenvalues of $Dv(p_j)$ differ from the eigenvalues of $Dv(p_k)$ only by multiplication of a common factor of $\sqrt{-1}$. 

\medskip
Besides making the genericity assumptions explicit in the Hamiltonian case we are also able to solve the realization problem.

\begin{definition}
A collection of four non-zero complex numbers $\mathcal{C}=\{d_1,\ldots,d_4\}$ is called \emph{admissible} for quadratic Hamiltonian vector fields if it satisfies
\[ \sum_{k=1}^4 \frac{1}{d_k} = 0. \]
An admissible collection is called \emph{exceptional} if there exist $d_j,d_k$ such that $d_j+d_k=0$.
\end{definition}

\begin{theorem}\label{thm:realization}
A collection $\mathcal{C}$ admissible for quadratic Hamiltonian vector fields is realizable as the spectrum of a quadratic Hamiltonian vector field if and only if one of the following two conditions holds:
\begin{enumerate}
 \item $\mathcal{C}$ is a non-exceptional collection,
 \item $\mathcal{C}$ is exceptional of the form $\{d,-d,d,-d\}$, for some $d\in\C^{\ast}$.
\end{enumerate}
\end{theorem}

\begin{proof}[Proof of Theorems \ref*{thm:twinHam} and \ref*{thm:realization}]
Every quadratic vector field with four isolated singularities is affine equivalent to a vector field with singularities at $p_1=(0,0)$, $p_2=(1,0)$, $p_3=(0,1)$. Any such vector field is completely determined by six parameters. Indeed, if $v=P\frac{\partial}{\partial x}+Q\frac{\partial}{\partial y}$ vanishes on such points then $P$ and $Q$ must be of the form
\begin{equation}\label{eq:coefficients}
  \arraycolsep=1pt\def\arraystretch{1.4}
\begin{array}{l} 
P(x,y) = a_1x^2+a_2xy+a_3y^2-a_1x-a_3y,  \\ 
Q(x,y) = a_4x^2+a_5xy+a_6y^2-a_4x-a_6y, \\
\end{array}
\end{equation}
for arbitrary $a_1,\ldots,a_6\in\C$. In the particular case the vector field in question is Hamiltonian we further obtain the following:
\begin{align*}
a_2 &= 2a_1, \\
a_5 &= -2a_1, \\
a_6 &= -a_1.
\end{align*}
A short computation shows that the determinant of $d_k:=Dv(p_k)$ is given by
\begin{equation}\label{eq:detsHam}
  \arraycolsep=1pt\def\arraystretch{1.4}
\begin{array}{l} 
d_1 = -a_1^2-a_3a_4, \\
d_2 = -a_1^2-2a_1a_4+a_3a_4, \\
d_3 = -a_1^2+2a_1a_3+a_3a_4. \\
\end{array}
\end{equation}
If the spectrum of $v$ is non-exceptional we can always solve the above system of equations for $a_1,a_3,a_4$ in terms of $d_1,d_2,d_3$ to obtain:
\begin{equation*}
  \arraycolsep=1pt\def\arraystretch{2.8}
\begin{array}{l} 
a_1^2 = \displaystyle\frac{-(d_1+d_2)(d_1+d_3)}{2(d_2+d_3)}, \\
a_3   = a_1+\displaystyle\frac{d_1+d_3}{2a_1}, \\
a_4   = -a_1-\displaystyle\frac{d_1+d_2}{2a_1}. \\
\end{array}
\end{equation*}
This proves that non-exceptional admissible collections are realizable. Note that these expressions yield two solutions (depending on the branch of the square root chosen for $a_1^2$) and one solution differs from the other by a sign, hence proving \hyperref[thm:twinHam]{Theorem \ref*{thm:twinHam}}.

\medskip
Now, assume $v$ is a Hamiltonian vector field with exceptional spectrum; without loss of generality we can assume $d_1+d_2=0$. Note that the Euler-Jacobi relation (\ref{eq:EJ1}) implies that $d_3+d_4=0$. From equations (\ref{eq:detsHam}) we deduce that either $a_1=0$ or $a_1+a_4=0$. If $a_1=0$ then we conclude that $d_3=d_2$ (and thus $d_4=d_1$) and if $a_1+a_4=0$ we conclude that $d_3=d_1$ (and thus $d_4=d_2$). In either case we see that exceptional spectra of quadratic Hamiltonian vector fields are always of the form $\mathcal{C}=\{d,-d,d,-d\}$, for some $d\in\C^{\ast}$. On the other hand any such collection $\mathcal{C}$ can be realized (choosing $d_3=d_2$) by setting $a_1=0$ and $a_3,a_4$ any complex numbers that satisfy $a_3a_4=d$, since equations (\ref{eq:detsHam}) become
\begin{equation*}
  \arraycolsep=1pt\def\arraystretch{1.2}
\begin{array}{l} 
d_1 = -a_3a_4, \\
d_2 = a_3a_4, \\
d_3 = a_3a_4. \\
\end{array}
\end{equation*}
\end{proof}

Note that the last argument in the above proof shows that, given an exceptional collection $\mathcal{C}=\{d,-d,d,-d\}$, there exits a one-dimensional family of quadratic Hamiltonian vector fields which are pairwise not affine equivalent and realize $\mathcal{C}$ as their spectra of singularities. This implies that a quadratic Hamiltonian vector field with non-degenerate singularities and exceptional spectrum has a one-dimensional family of twin vector fields.

\section{Characteristic numbers at infinity}

It follows from \hyperref[thm:main]{Theorem \ref*{thm:main}} that quadratic vector fields are determined by their spectra of singularities up to finite ambiguity. This suggests that if we take into account additional analytic invariants we should expect to have enough information to single out a unique vector field. Thus far, we have not yet considered the singular points at infinity. For these singular points we cannot define unambiguously their spectra but we can always define their characteristic numbers. There is, however, a slight difficulty we will find: suppose we are given polynomials $P,Q$ defining a vector field, in order to compute the characteristic numbers at infinity we must firs \emph{find} the singular points at infinity; this is involves solving a cubic equation. In order to bypass this we will work backwards: we will assume that we are given the position and characteristic numbers of the singularities at infinity and use this information to recover the polynomials $P$ and $Q$ (subject to the normalization given by requiring the vector field to have singularities at $p_1=(0,0)$, $p_2=(1,0)$, $p_3=(0,1)$).

The following proposition is a well-known fact and its proof is immediate.

\begin{proposition}
In the generic case there is a one to one correspondence between quadratic homogeneous vector fields 
\[ \xi(x,y)=P_2\frac{\partial}{\partial x}+Q_2\frac{\partial}{\partial y}, \]
modulo multiplication by a non-zero complex number and foliations with linear monodromy 
\[ \frac{dz}{dw}=z\sum_{j=1}^3\frac{\mu_j}{w-w_j}, \qquad \mu_1+\mu_2+\mu_3=1, \]
defined in a neighborhood of the line at infinity $z=0$ (where $z,w$ are the coordinates $(z,w)=(1/x,y/x)$).
\end{proposition}

In virtue of the above proposition we can uniquely define a quadratic homogeneous vector field by specifying three points $w_1,w_2,w_3\in\P^1$, three numbers $\mu_1,\mu_2,\mu_3\in\C$ satisfying $\mu_1+\mu_2+\mu_3=1$, and a rescaling factor $\kappa\in\C^*$. Once such homogeneous vector field is given we use formula (\ref{eq:coefficients}) to define uniquely a quadratic vector field having singularities at $p_1=(0,0)$, $p_2=(1,0)$, $p_3=(0,1)$. 

A straightforward computation provides the following explicit formulae:

\[ \begin{array}{l}
a_1 =  \kappa\left(\mu_1 w_2 w_3+\mu_2 w_1 w_3+\mu_3 w_1 w_2\right),\\[1mm]
a_2 = -\kappa(\mu_1(w_2+w_3)+\mu_2(w_1+w_3)+\mu_3(w_1+w_2)),\\[1mm]
a_3 =  \kappa,\\[1mm]
a_4 =  \kappa w_1w_2w_3,\\[1mm]
a_5 = -\kappa(\mu_1w_1(w_2+w_3)+\mu_2w_2(w_1+w_3)+\mu_3w_3(w_1+w_2)),\\[1mm]
a_6 =  \kappa(\mu_1w_1+\mu_2w_2+\mu_3w_3).\\[1mm]
\end{array} \]

In order to prove that a generic vector field is completely determined by the spectra of its finite singularities and the characteristic numbers at infinity we will define a \emph{moduli map} $\mathcal{M}$ that assigns to each vector field its set of spectra and characteristic numbers.

In virtue of \hyperref[lemma]{Lemma \ref*{lemma}}, the spectra of $v$ is completely determined by the spectra at $p_1,p_2,p_3$. Let us define a map $\spec{}{}\colon\C^6\to\C^6$ that assigns to each vector field as above its spectra of singularities at $p_1,p_2,p_3$. Namely,
\begin{equation}\label{eq:modulimap1}
\spec{}{(a_1,\ldots,a_6)}=(\spec{}{p_1},\spec{}{p_2},\spec{}{p_3}).
\end{equation}
Note that the map $\spec{}{}$ is polynomial and its components are either linear (for the traces) or quadratic (for the determinants).

\hyperref[thm:main]{Theorem \ref*{thm:main}} may now be restated as follows.

\begin{thmbis}\hypertarget{thm:bis}{}
The map $\spec{}{}\colon\C^6\to\C^6$ is a regular dominant map whose generic fiber consists of two points.
\end{thmbis}

We now define the moduli map $\mathcal{M}\colon\C^6\to\C^2\times\C^6$ by the formula
\[\mathcal{M}(\mu_1,\mu_2,w_1,w_2,w_3,\kappa)=((\mu_1,\mu_2),\spec{}{(a_1,\ldots,a_6)}). \]

In order to prove \hyperref[thm:charnumbers]{Theorem \ref*{thm:charnumbers}} we need to show that the generic fiber of the moduli map $\mathcal{M}$ consists of a single point. In order to do this it is enough to find a non-empty open set $W\subset\operatorname{Im }\mathcal{M}$ and $U\subset\C^6$ such that $U=\mathcal{M}^{-1}(W)$ and $\mathcal{M}\vert_{U}\colon U\to W$ is one-to-one.

\begin{proof}[Proof of Theorem \ref*{thm:charnumbers}]
Consider the following vector field
\[ v = (x^2+2xy-x)\frac{\partial}{\partial x}+(-xy+3y^2-3y)\frac{\partial}{\partial y}. \]
A simple computation shows that the unique twin vector field of $v$ is
\[ \tilde{v} = (3x^2+6xy-3x)\frac{\partial}{\partial x}+\left(-\frac{7}{3}x^2-5xy+y^2+\frac{7}{3}x-y\right)\frac{\partial}{\partial y}, \]
and that the derivative of the map $\spec{}{}$ is invertible both at $v$ and at $\tilde{v}$. Let $\Lambda=\spec{}{v}$. We can deduce from the inverse function theorem and from \hyperlink{thm:bis}{Theorem 1.3'} the existence of neighborhoods $W_1,U,\widetilde{U}$ of $\Lambda,v,\tilde{v}$ respectively such that $\spec{}{}^{-1}(W_1)=U\cup\widetilde{U}$ and $\spec{}{}$ maps both $U$ and $\widetilde{U}$ diffeomorphically onto $W_1$. On the other hand, it is not hard to check that the characteristic numbers at infinity of $v$ and $\tilde{v}$ are different (the characteristic numbers of $v$ are all rational whereas the characteristic numbers of $\tilde{v}$ are not). By shrinking $U$ and $\widetilde{U}$ if necessary we can assume that vector fields in $U$ have different characteristic numbers from any vector field in $\widetilde{U}$. This means that if $(\mu_1,\mu_2)$ are the characteristic numbers of $v$ we can find a small neighborhood $W_0$ of $(\mu_1,\mu_2)$ such that no vector field from $\widetilde{U}$ has characteristic numbers in $W_0$. Define $W=W_0\times W_1\subset\C^2\times\C^6$. Since $\mathcal{M}=(\mu_1,\mu_2,\spec{}{})$ and $\spec{}{}^{-1}(W_1)=U\cup\widetilde{U}$, we must have $\mathcal{M}^{-1}(W)\subset U\cup\widetilde{U}$. However we also know that $\mathcal{M}^{-1}(W)$ is disjoint from $\widetilde{U}$, by construction of $W_0$, and so $\mathcal{M}^{-1}(W)=U$. We conclude that $\mathcal{M}$ is one-to-one over $W$ and so the generic fiber consists of a single point.
\end{proof}

\section{The image of the moduli map \texorpdfstring{$\mathcal{M}$}{the moduli map}}

As shown by \hyperref[thm:charnumbers]{Theorem \ref*{thm:charnumbers}}, the closure of the image of the moduli map $\mathcal{M}\colon\C^6\to\C^8$ is a codimension two subvariety of $\C^8$. This implies that there are at least two independent relations between the spectra of finite singularities and the characteristic numbers at infinity. One of these is the well-known \emph{Baum-Bott formula,} yet there must exist one more relation.

\begin{question}\label{ques:hiddenvf}
What is the missing relation between the spectra of finite singularities and the characteristic numbers at infinity for a generic quadratic vector field?
\end{question}

The above question is closely related to a question on the \emph{hidden relations} between the spectra of the derivatives at the fixed points of a regular endomorphism $f\colon\P^2\to\P^2$, posed by Adolfo Guillot in \cite{Guillot2004}. There are known relations (generalizing the \emph{holomorphic Lefschetz fixed point formula}) among these eigenvalues yet a dimensional argument shows that there must be even more relations. Guillot's question is: what are those missing relations?

\section{Acknowledgments}

The proof of \hyperref[thm:main]{Theorem \ref*{thm:main}} was improved thanks to arguments suggested by Yulij S.~Ilyashenko and John H.~Hubbard. 
The results presented in \hyperref[sec:QHVFs]{$\mathsection$\ref*{sec:QHVFs}} were a collaboration with Gilberto Bruno and Jessica Jaurez at Instituto de Matem\'{a}ticas UNAM. I would like to thank them for their interest in the problem and for their valuable contribution. My deepest gratitude to Laura Ortiz for the invitation to the visit where these discussions took place.
I'm thankful to Adolfo Guillot for introducing me to his work on endomorphisms of the projective plane. I also want to thank Adolfo Guillot and Yulij S.~Ilyashenko for their comments on earlier versions of this paper. 

\bibliographystyle{alpha}
\bibliography{ref-twin_vfs.bib}

\begin{thebibliography}{Gui04}

\bibitem[GH94]{GriffithsHarris1994}
Phillip Griffiths and Joseph Harris.
\newblock {\em Principles of algebraic geometry}.
\newblock Wiley Classics Library. John Wiley \& Sons, Inc., New York, 1994.
\newblock Reprint of the 1978 original.
  \href{http://dx.doi.org/10.1002/9781118032527}{DOI:10.1002/9781118032527}.

\bibitem[Gui04]{Guillot2004}
A.~Guillot.
\newblock Un th\'eor\`eme de point fixe pour les endomorphismes de l'espace
  projectif avec des applications aux feuilletages alg\'ebriques.
\newblock {\em Bull. Braz. Math. Soc. (N.S.)}, 35(3):345--362, 2004.
\newblock
  \href{http://dx.doi.org/10.1007/s00574-004-0018-7}{DOI:10.1007/s00574-004-0018-7}.

\bibitem[IM11]{IlyashenkoMoldavskis2011}
Yu.~S. Ilyashenko and V.~Moldavskis.
\newblock Total rigidity of generic quadratic vector fields.
\newblock {\em Mosc. Math. J.}, 11(3):521--530, 630, 2011.
\newblock \href{http://www.ams.org/mathscinet-getitem?mr=2894428}{MR:2894428}.

\bibitem[LN12]{LinsNeto2012}
A.~Lins~Neto.
\newblock Fibers of the {B}aum-{B}ott map for foliations of degree two on
  {${\mathbb{P}}^2$}.
\newblock {\em Bull. Braz. Math. Soc. (N.S.)}, 43(1):129--169, 2012.
\newblock
  \href{http://dx.doi.org/10.1007/s00574-012-0008-0}{DOI:10.1007/s00574-012-0008-0}.

\end{thebibliography}

\end{document}